\documentclass[graybox]{svmult}

\usepackage{mathptmx}       
\usepackage{helvet}         
\usepackage{courier}        
\usepackage{type1cm}        
\usepackage{makeidx}         
\usepackage{graphicx}        
\usepackage{subfigure}         
\usepackage{multicol}        
\usepackage[bottom]{footmisc}
\usepackage{amsmath}
\usepackage{amsfonts}
\usepackage{amssymb}
\usepackage{xspace}
\usepackage{fullpage}
\usepackage{natbib}

\makeindex             

\newcommand{\loco}{{\sc{Loco}}\xspace}

\newcommand{\samp}{n}
\newcommand{\subsamp}{m}
\newcommand{\dims}{p}
\newcommand{\dimss}{d}
\newcommand{\x}{\vec{X}}
\newcommand{\pp}{\vec{\phi}}
\newcommand{\cls}{\hat{\beta}_d^{\vec{\phi}}}

\begin{document}

\title*{Random Projections For Large-Scale Regression}
\author{Gian-Andrea Thanei, Christina Heinze and Nicolai Meinshausen}
\institute{Gian-Andrea Thanei \at ETH Z\"urich,  R\"amistrasse 101
8092 Z\"urich, Switzerland, \email{thanei@stat.math.ethz.ch} \and 
Christina Heinze \at ETH Z\"urich, R\"amistrasse 101
8092 Z\"urich, Switzerland, \email{heinze@stat.math.ethz.ch} \and
Nicolai Meinshausen \at ETH Z\"urich,  R\"amistrasse 101
8092 Z\"urich, Switzerland, \email{meinshausen@stat.math.ethz.ch}}

\maketitle

\abstract{Fitting linear regression models can be computationally very
  expensive in large-scale data analysis tasks if the sample size and
  the number of variables are very large. 
Random projections are extensively used as a dimension reduction tool
in machine learning and statistics. We discuss the applications of
random projections in linear regression problems, developed to decrease computational costs, and give an overview of the
theoretical guarantees of the generalization error.
 It can be shown that the combination of random projections with least
 squares regression leads to similar recovery as ridge regression and principal
 component regression. We also discuss possible improvements when
 averaging over multiple random projections, an approach that lends
 itself easily to parallel implementation.}

\section{Introduction}
\label{sec:1}

Assume we are given a data matrix $\x \in \mathbb{R}^{n \times p}$ ($n$ samples of a $p$-dimensional random variable) and a response vector $\vec{\vec{Y}} \in \mathbb{R}^n$. We assume a linear model for the data where $\vec{Y}=\vec{X}\beta+\varepsilon$ for some regression coefficient  $\beta \in \mathbb{R}^p$ and $\varepsilon$  i.i.d. mean-zero noise. Fitting a regression model by standard least squares or ridge regression requires $\mathcal{O}(n p^2)$ or $\mathcal{O}(p^3)$ flops. In the situation of large-scale ($n, p$ very large) or high dimensional ($p \gg n$) data these algorithms are not applicable without having to pay a huge computational price. 

Using a random projection, the data can be ``compressed'' either row- or column-wise. Row-wise compression was proposed and discussed in \citet{Zhou:2007,Dhillon:2013wz, McWilliams:2014}. These approaches replace the  least-squares estimator 
\begin{equation}\label{eq:pre} \underset{\gamma \in \mathbb{R}^p}{\operatorname{argmin}} \| \vec{Y} - \x \gamma \|_2^2 \qquad\mbox{ with the estimator } \qquad \underset{\gamma \in \mathbb{R}^p}{\operatorname{argmin}} \| \vec{\psi} \vec{Y} - \vec{\psi} \x \gamma \|_2^2 ,\end{equation}
where the matrix 
$\vec{\psi} \in \mathbb{R}^{m \times n}$ ($m \ll n$) is a random projection matrix and has, for example, i.i.d. $\mathcal{N}(0, 1)$ entries. Other possibilities for the choice of $\vec{\psi}$ are discussed below. The high-dimensional setting and $\ell_1$-penalized regression is considered in \citet{Zhou:2007}, where it is shown that a sparse linear model can be recovered from the projected data  under certain conditions.  The optimization problem is  still $p$-dimensional, however,  and computationally expensive if the number of variables is very large.  

Column-wise compression addresses this later issue by reducing the problem to a $d$-dimensional optimization with  $d\ll p$ by replacing the least-squares estimator 
\begin{equation}\label{eq:post} \underset{\gamma \in \mathbb{R}^p}{\operatorname{argmin}} \| \vec{Y} - \x \gamma \|_2^2 \qquad\mbox{ with the estimator } \qquad \vec{\vec{\phi}} \; \underset{\gamma \in \mathbb{R}^d}{\operatorname{argmin}}\| \vec{Y}-\vec{X}\vec{\phi} \gamma \|_2^2, \end{equation}
where the random projection matrix is now $\vec{\phi} \in \mathbb{R}^{p \times d}$ (with $d \ll p$). 
By right multiplication to the data matrix $\x$ we transform the data matrix to $\x \vec{\phi} $ and thereby reduce the number of variables from $p$ to $d$ and thus reducing computational complexity. The Johnson-Lindenstrauss Lemma \citep{jl-orig,jl-gauss-proof,jl-gauss-proof2} guarantees that the distance between two transformed sample points is approximately preserved in the column-wise compression.

Random projections have also been considered under the aspect of preserving privacy ~\citep{Blocki:2012}. By pre-multiplication with a random projection matrix as in~\eqref{eq:pre} no observation in the resulting matrix can be identified with one of the original data points. Similarly, post-multiplication as in~\eqref{eq:post} produces new variables that do not reveal the realized values of the original variables.

In many applications the random projection used in practice falls under the class of Fast Johnson-Lindenstrauss Transforms (FJLT) \citep{FJLT}. One instance of such a fast projection is the Subsampled Randomized Hadamard Transform (SRHT) \citep{Tropp:2010uo}. Due to its recursive definition, the matrix-vector product has a complexity of $\mathcal{O}(\dims \log(\dims))$, reducing the cost of the projection to  $\mathcal{O}(\samp \dims \log(\dims))$.
Other proposals that lead to speedups compared to a Gaussian random projection matrix include random sign or sparse random projection matrices \citep{achlioptas2003}. Notably, if the data matrix is sparse, using a sparse random projection can exploit sparse matrix operations. Depending on the number of non-zero elements in $\x$, one might prefer using a sparse random projection over a FJLT that cannot exploit sparsity in the data. 
Importantly, using $\x\pp$ instead of $\x$ in our regression algorithm of choice can be disadvantageous if $\x$ is extremely sparse and $\dimss$ cannot be chosen to be much smaller than $\dims$. (The projection dimension $\dimss$ can be chosen by cross validation.) As the multiplication by $\pp$ ``densifies'' the design matrix used in the learning algorithm the potential computational benefit of sparse data is not preserved.

For OLS and row-wise compression as in~\eqref{eq:pre}, where $\samp$ is very large and $ \dims < \subsamp < \samp$, the SRHT (and similar FJLTs) can be understood as a subsampling algorithm. It preconditions the design matrix by rotating the observations to a basis where all points have approximately uniform leverage \citep{Dhillon:2013wz}. This justifies uniform subsampling in the projected space which is applied subsequent to the rotation in order to reduce the computational costs of the OLS estimation. Related ideas can be found in the way columns and rows of $\x$ are sampled in a CUR-matrix decomposition \citep{mahoney2009cur}.  While the approach in \citet{Dhillon:2013wz} focuses on the concept of leverage, \citet{McWilliams:2014} propose an alternative scheme that allows for outliers in the data and makes use of the concept of influence \citep{Cook:1977}. Here, random projections are used to approximate the influence  of each observation which is then used in the subsampling scheme to determine which observations to include in the subsample.

Using random projections column-wise as in~\eqref{eq:post} as a dimensionality reduction technique in conjunction with ($\ell_2$ penalized) regression has been considered in \citet{Lu:2013}, \citet{kaban:2014} and \citet{clsr}. The main advantage of these algorithms is the computational speedup while preserving predictive accuracy. Typically, a variance reduction is traded off against an increase in bias.
In general, one disadvantage of reducing the dimensionality of the data is that the coefficients in the projected space are not interpretable in terms of the original variables. Naively, one could reverse the random projection operation by projecting the coefficients estimated in the projected space back into the original space as in~\eqref{eq:post}. For prediction purposes this operation is irrelevant, but  it can be shown that this estimator does not approximate the optimal solution in the original $p$-dimensional coefficient space well \citep{zhang2012recovering}. As a remedy, \cite{zhang2012recovering} propose to find the dual solution in the projected space to recover the optimal solution in the original space. The proposed algorithm approximates the solution to the original problem accurately if the design matrix is low-rank or can be sufficiently well approximated by a low-rank matrix.

Lastly, random projections have been used as an auxiliary tool. As an example, the goal of \cite{mcwilliams2014loco} is to  distribute ridge regression across variables with an algorithm called \loco. The design matrix is split across variables and the variables are distributed over processing units (workers). Random projections are used to preserve the dependencies between all variables in that each worker uses a randomly projected version of the variables residing on the other workers in addition to the set of variables assigned to itself. It then solves a ridge regression using this local design matrix. The solution is the concatenation of the coefficients found from each worker and the solution vector lies in the original space so that the coefficients are interpretable. Empirically, this scheme achieves large speedups while retaining good predictive accuracy. Using some of the ideas and results outlined in the current manuscript, one can show that the difference between the full solution and the coefficients returned by \loco is bounded.


Clearly, row- and column-wise compression can also be applied simultaneously or column-wise compression can be used together with subsampling of the data instead of row-wise compression. In the remaining sections, we will focus on the column-wise compression as it poses more difficult challenges in terms of statistical performance guarantees. While row-wise compression just reduces the effective sample size and can be expected to work in general settings as long as the compressed dimension $m<n$ is not too small \citep{Zhou:2007}, column-wise compression can only work well if certain conditions on the data are satisfied and we will give an overview of these results. If not mentioned otherwise, we will refer with compressed regression and random projections to the column-wise compression.

The structure of the manuscript is as follows:
We will give an overview of bounds on the estimation accuracy in the following section \ref{sec:2}, including both known results and new contributions in the form of tighter bounds. In Section \ref{sec:3} we will discuss the possibility and properties of variance-reducing averaging schemes, where estimators based on different realized random projections are aggregated. Finally, Section \ref{sec:discuss} concludes the manuscript with a short discussion.


\section{Theoretical Results}
\label{sec:2}
We will discuss in the following the properties of the column-wise
compressed estimator as in~\eqref{eq:post}, which is defined as 
\begin{equation}\label{eq:post2}
\cls\; =\; \vec{\vec{\phi}} \; \underset{\gamma \in \mathbb{R}^d}{\operatorname{argmin}}\| \vec{Y}-\vec{X}\vec{\phi} \gamma \|_2^2,
\end{equation}
where we assume that $\phi$ has i.i.d. $\mathcal{N}(0,1/d)$ entries. This estimator will be referred to as the
compressed least squares estimator (CLSE) in the following. We will focus on
the unpenalized form as in~\eqref{eq:post2} but note that similar
results also apply to estimators that put an additional penalty on the
coefficients $\beta$ or $\gamma$. Due to the isotropy of the random
projection, a ridge-type penalty as in \citet{Lu:2013} and \citet{mcwilliams2014loco} is
perhaps a natural choice. An interesting summary of the bounds on
random projections is on the other hand that the random projection as
in~\eqref{eq:post2} already acts as a regularization  and the
theoretical properties of~\eqref{eq:post2} are very much related to
the properties of a ridge-type estimator of the coefficient vector in
the absence of random projections.   

We will restrict discussion of the properties mostly to the mean squared error (MSE) 
\begin{equation}
\label{mse}
\mathbb{E}_{\vec{\phi}}\big[\mathbb{E}_{\varepsilon}(\|\vec{X} \beta -\vec{X} \cls \|_2^2)\big].
\end{equation}
First results on compressed least squares have been given
in~\citep{clsr}  in a random design setting. It was shown that the
bias of the estimator~\eqref{eq:post2}  is of order $\mathcal{O}(\log(n)/d)$. This proof
used a modified version of the Johnson-Lindenstrauss Lemma.  A recent result~\citep{kaban:2014} shows that the $\log(n)$-term
is not necessary for fixed design settings where
$\vec{Y}=\vec{X}\beta+\varepsilon$ for some $\beta \in \mathbb{R}^p$
and $\varepsilon$ is i.i.d. noise, centered
$\mathbb{E}_{\varepsilon}[\varepsilon]=0$ and with the variance
$\mathbb{E}_{\varepsilon}[\varepsilon \varepsilon']=\sigma^2 I_{n
  \times n}$. We will work with this setting in the following. 

The following result of~\citep{kaban:2014} gives a bound on the MSE for
fixed design.
\begin{theorem}\label{clsmse}~\citep{kaban:2014} Assume fixed design and $\mathrm{Rank}(\vec{X}) \geq d$. Then 
\begin{equation} 
\mathbb{E}_{\vec{\phi}}\big[\mathbb{E}_{\varepsilon}(\|\vec{X} \beta -\vec{X} \cls \|_2^2)\big]\;\leq \;\sigma^2 d + \frac{\| \x \beta\|_2^2}{d}+\operatorname{trace}(\x^{\prime}\x)\frac{\| \beta \|_2^2}{d}.
\end{equation}
\end{theorem}
\begin{proof}
See Appendix.
\end{proof}
Compared with~\citep{clsr}, the result removes an unnecessary
$\mathcal{O}(\log(n))$ term and demonstrates the $\mathcal{O}(1/d)$
behaviour of the bias. 
The result also illustrates the tradeoffs when choosing a suitable
dimension $d$ for the projection. Increasing $d$ will lead to a $1/d$
reduction in the bias terms but lead to a linear increase in the
estimation error (which is proportional to the dimension in which the
least-squares estimation is performed). An optimal bound can only be
achieved with a value of $d$ hat depends on the unknown signal and in
practice one would typically use cross-validation to make the choice
of the dimension of the projection.

One issue with the  bound in Theorem~\ref{clsmse} is that the bound on the bias term in the noiseless case ($Y=\x \beta$)
\begin{equation}
\mathbb{E}_{\vec{\phi}}\big[\mathbb{E}_{\varepsilon}(\|\vec{X} \beta -\vec{X} \cls \|_2^2)\big]\;\leq \frac{\| \x \beta\|_2^2}{d}+\operatorname{trace}(\x^{\prime}\x)\frac{\| \beta \|_2^2}{d}
\end{equation}
is usually weaker than the trivial bound (by setting $\cls = 0 $) of
\begin{equation}
\mathbb{E}_{\vec{\phi}}\big[\mathbb{E}_{\varepsilon}(\|\vec{X} \beta -\vec{X} \cls \|_2^2)\big]\;\leq \|\vec{X} \beta\|_2^2
\end{equation}
for most values of $d < p$. By improving the bound, it is also possible to point out the similarities between ridge regression and compressed least squares.

The improvement in the bound rests on a small modification in the original proof in~\citep{kaban:2014}. The idea is to bound the bias term of (\ref{mse}) by optimizing over the upper bound given in the foregoing theorem. Specifically, one can use the inequality 
\begin{align*}
\mathbb{E}_{\vec{\phi}}[\mathbb{E}_{\varepsilon}[\| \x \beta -\x\vec{\phi} (\vec{\phi}^{\prime}
\x^{\prime} \x \vec{\phi})^{-1} \vec{\phi}^{\prime} \x^{\prime} \x
\beta \|_2^2]]\;&\leq\; \min_{\hat{\beta} \in
  \mathbb{R}^p}\; \mathbb{E}_{\vec{\phi}}[\mathbb{E}_{\varepsilon}[\| \x \beta -\x\vec{\phi}\vec{\phi}^{\prime} \hat{\beta} \|_2^2]], \\
\textrm{instead of }\hspace{8cm}  \\
\mathbb{E}_{\vec{\phi}}[\mathbb{E}_{\varepsilon}[\| \x \beta -\x\vec{\phi} (\vec{\phi}^{\prime}
\x^{\prime} \x \vec{\phi})^{-1} \vec{\phi}^{\prime} \x^{\prime} \x
\beta \|_2^2]]\;&\leq\; \mathbb{E}_{\vec{\phi}}[\mathbb{E}_{\varepsilon}[\| \x \beta -\x\vec{\phi}\vec{\phi}^{\prime} \beta \|_2^2]]. 
\end{align*}
To simplify the exposition we will from now on always assume we have rotated the design matrix
to an orthogonal design so that the Gram matrix is diagonal:
\begin{equation}
\Sigma=\x^{\prime} \x=\textit{diag}(\lambda_1,...,\lambda_p).
\end{equation}
This can always be achieved for any design matrix and is thus not a restriction. It implies,
however, that the optimal regression coefficients $\beta$ are
expressed in the basis in which the Gram matrix is orthogonal, this is the basis of principal components. This will
turn out to be the natural choice for random projections and allows for easier interpretation
of the results. \\
Furthermore note that in Theorem 1 we have the assumption $\mathrm{Rank}(\vec{X}) \geq d$, which tells us that we can apply the CLSE in the high dimensional setting $p \gg n$ as long as we choose $d$ small enough (smaller than $\mathrm{Rank}(\vec{X})$, which is usually equal to $n$) in order to have uniqueness. \\
With the foregoing discussion on how to improve the bound in Theorem 1 we get the following theorem:
\begin{theorem}\label{clrmseimp} Assume $\mathrm{Rank}(\x) \geq d$, then the MSE (\ref{mse}) can be bounded above by
\begin{equation} \label{eq:varboundrandom}
\mathbb{E}_{\vec{\phi}}[\mathbb{E}_{\varepsilon}[\|\x \beta -\x \cls \|_2^2]] \leq \sigma^2 d +\sum_{i=1}^p \beta_i^2 \lambda_i w_i
\end{equation}
where 
\begin{equation}
w_i = \frac{(1+1/d)\lambda_i^2+(1+2/d)\lambda_i \operatorname{trace}(\Sigma)+\operatorname{trace}(\Sigma)^2/d}{(d+2+1/d)\lambda_i^2+2(1+1/d)\lambda_i \operatorname{trace}(\Sigma)+\operatorname{trace}(\Sigma)^2/d}.\\
\end{equation}
\end{theorem}
\begin{proof}
See Appendix.
\end{proof}
The $w_i$ are shrinkage factors. By defining the proportion of the total variance observed in the direction of the $i$-th principal component as
\begin{equation}
\alpha_i=\frac{\lambda_i}{\operatorname{trace}(\Sigma)},
\end{equation}
we can rewrite the shrinkage factors in the foregoing theorem as
\begin{equation}
w_i=\frac{(1+1/d)\alpha_i^2+(1+2/d)\alpha_i +1/d}{(d+2+1/d)\alpha_i^2+2(1+1/d)\alpha_i +1/d}.
\end{equation}
Analyzing this term shows that the shrinkage is stronger in directions of high variance compared to directions of low variance. To explain this relation in a bit more detail we compare it to ridge regression. The MSE of ridge regression with penalty term $\lambda \|\beta\|_2^2$ is given by
\begin{equation}\label{eq:varboundridge}
\mathbb{E}_{\varepsilon} [\|\x \beta-\x \beta^{Ridge}\|_2^2]= \sigma^2 \sum_{i=1}^p \Big( \frac{\lambda_i}{\lambda_i+\lambda}\Big)^2+\sum_{i=1}^p \beta_i^2 \lambda_i \Big( \frac{\lambda}{\lambda+\lambda_i} \Big)^2. 
\end{equation}
Imagine that the signal lives on the space spanned by the first $q$ principal directions, that is $\beta_i=0$ for $i>q$. The best MSE we could then achieve is $\sigma^2 q$ by running a regression on the first $q$ first principal directions. For random projections, we can see that we can indeed reduce the bias term to nearly zero by forcing $w_i\approx 0$ for $i=1,\ldots,q$. This requires $d\gg q$ as the bias factors will then vanish like $1/d$.
Ridge regression on the other hand requires that the penalty $\lambda$ is smaller than the $q$-th largest eigenvalue $\lambda_q$ (to reduce the bias on the first $q$ directions) but large enough to render the variance factor $\lambda_i/(\lambda_i + \lambda) $ very small for $i>q$. The tradeoff in choosing the penalty $\lambda$ in ridge regression and choosing the dimension $d$ for random projections is thus very similar. The number of directions for which the eigenvalue $\lambda_i$ is larger than the penalty $\lambda$ in ridge corresponds to the effective dimension and will yield the same variance bound as in random projections. 
 The analogy between the MSE bounds~\eqref{eq:varboundrandom} for random projections and~\eqref{eq:varboundridge} for ridge regression illustrates thus a close relationship between compressed least squares and ridge regression or principal component regression, similar to~\citet{Dhillon:2013tw}. 

Instead of an upper bound for the MSE  of CLSE as in  \citet{clsr} and \citet{kaban:2014}, we will in the following try to derive 
explicit expressions for the MSE, following the ideas in \citet{kaban:2014} and \citet{marzetta:2013} and we give a closed form MSE in the case of orthonormal predictors. The derivation will make use of the following notation:
\begin{definition}
Let $\vec{\phi} \in \mathbb{R}^{p \times d}$ be a random projection. We define the following matrices:
\begin{align*}
\vec{\phi}_d^{\x}=&\vec{\phi} (\vec{\phi}^{\prime} \x^{\prime} \x \vec{\phi} )^{-1} \vec{\phi}^{\prime} \in \mathbb{R}^{p \times p} \quad 
\textrm{and}\quad 
T_d^{\vec{\phi}}=\mathbb{E}_{\vec{\phi}}[\vec{\phi}_d^{\x}]=\mathbb{E}_{\vec{\phi}}[\vec{\phi} (\vec{\phi}^{\prime} \x^{\prime} \x \vec{\phi} )^{-1} \vec{\phi}^{\prime}] \in \mathbb{R}^{p \times p}.
\end{align*}
\end{definition}
The next Lemma~\citep{marzetta:2013} summarizes the main properties of $\vec{\phi}_d^{\x}$ and $T_d^{\vec{\phi}}$.
\begin{lemma}
Let $\vec{\phi} \in \mathbb{R}^{p \times d}$ be a random projection. Then 
\begin{itemize}
\item[$i)$] $\,\,\,(\vec{\phi}_d^{\x})^{\prime}=\vec{\phi}_d^{\x}$ (symmetric),
\item[$ii)$] $\,\,\,\vec{\phi}_d^{\x} \x^{\prime} \x \vec{\phi}_d^{\x} = \vec{\phi}_d^{\x}$ (projection),
\item[$iii)$] $\,\,$ if $\Sigma=\x^{\prime}\x$ is diagonal $\Rightarrow$ $T_d^{\vec{\phi}}$ is diagonal.
\end{itemize}
\end{lemma}
\begin{proof}
See~\citet{marzetta:2013}.
\end{proof}
The important point of this lemma is that when we assume orthogonal design then $T_d^{\vec{\phi}}$ is diagonal. We will denote this by
\begin{equation*}
T_d^{\vec{\phi}}=\mathrm{diag}(1/\eta_1,...,1/\eta_p),
\end{equation*}
where the terms $\eta_i$ are well defined but without an explicit representation.

\begin{figure}[t]
\centerline{\mbox{\subfigure{\includegraphics[width=6cm]{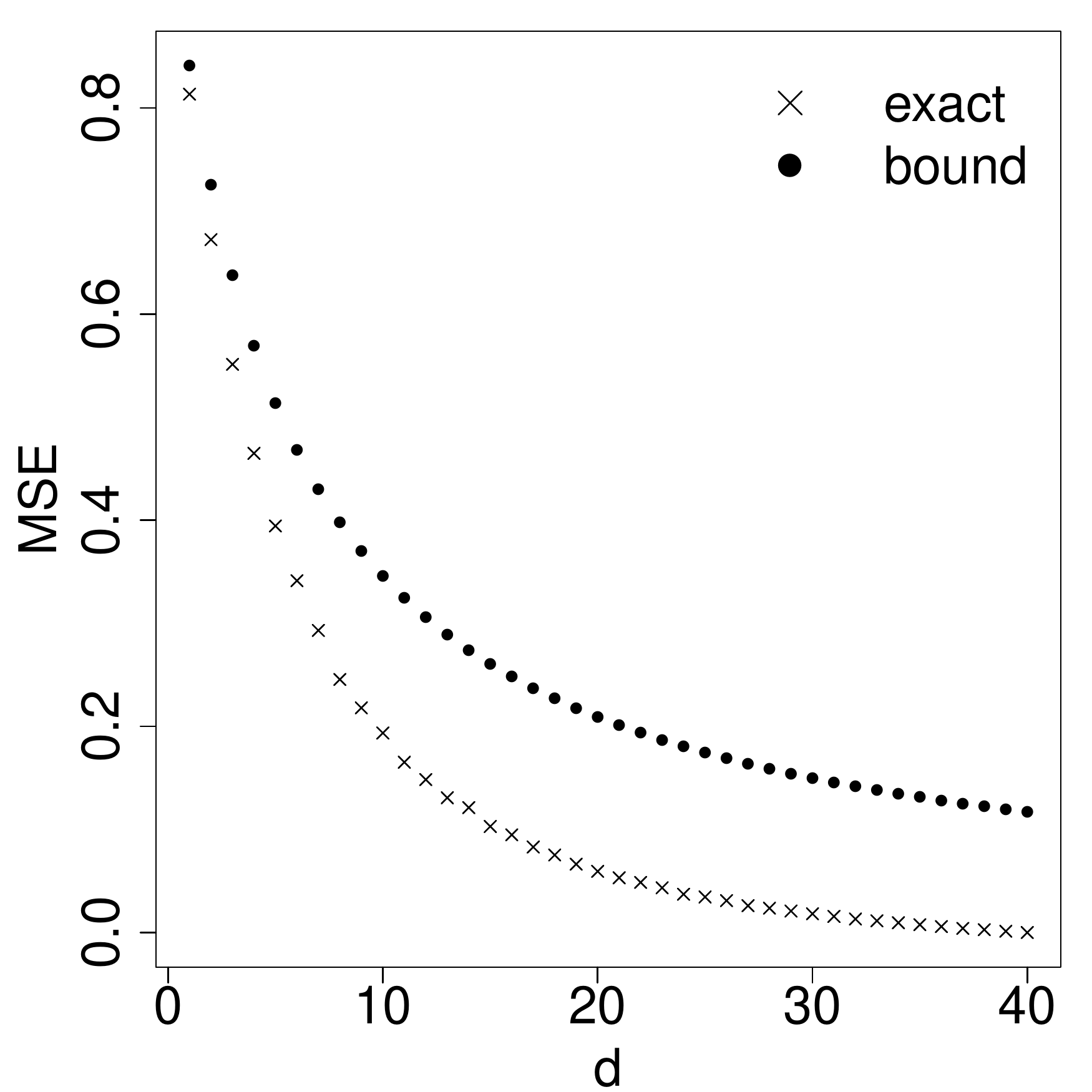}}\quad
\subfigure{\includegraphics[width=6cm]{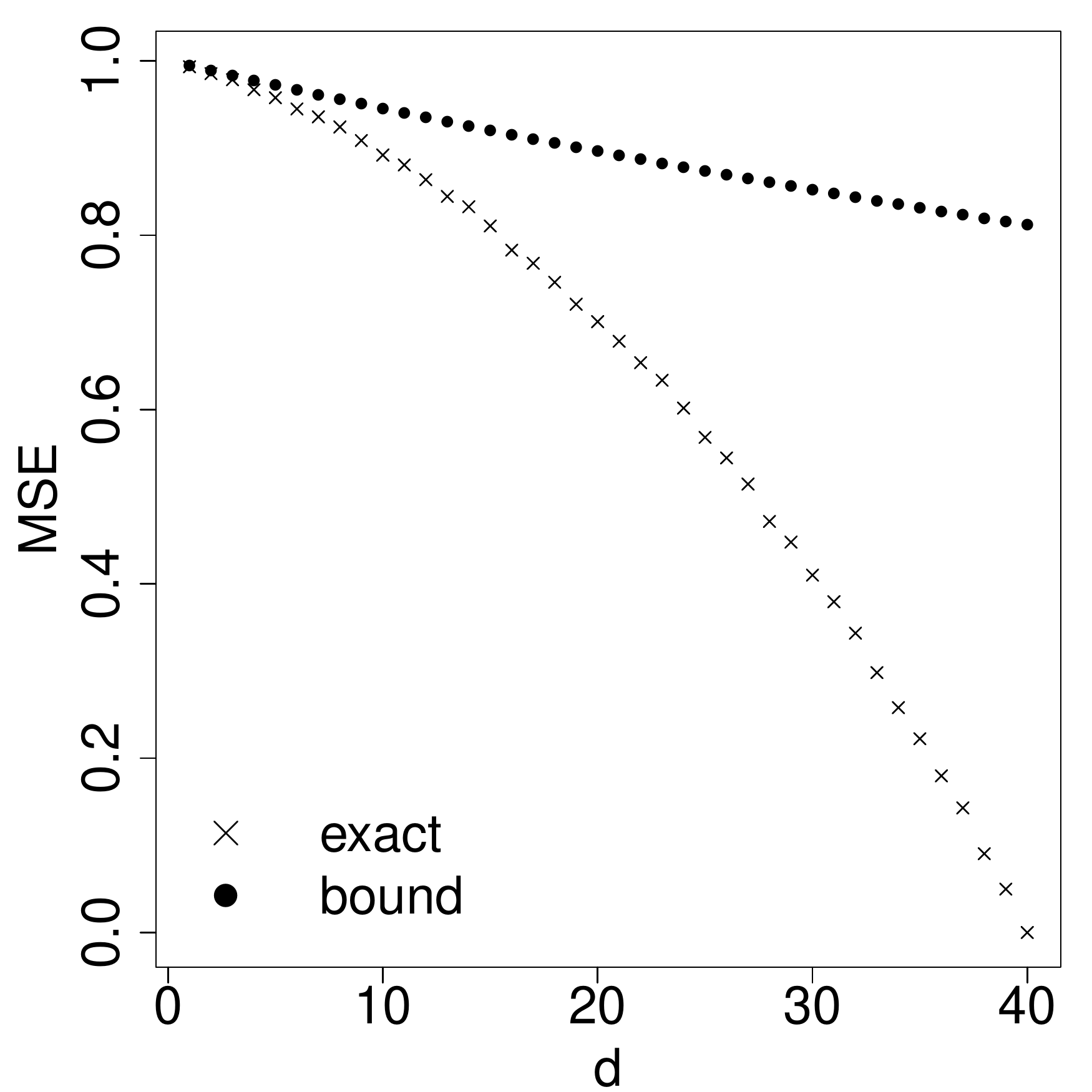} }}}
\caption{Numerical simulations of the bounds in Theorems 2 and 3. Left: the exact factor $(1-\lambda_1/\eta_1)$ in the MSE  is plotted versus the bound $w_1$ as a function of the projection dimension $d$. Right:  the exact factor $(1-\lambda_p/\eta_p)$ in the MSE and the upper bound $w_p$. Note that
  the upper bound works especially well for small values of $d$ and for the larger eigenvalues and
 is always below the trivial bound $1$.} \label{fig1}
\end{figure}

A quick calculation reveals the following theorem:
\begin{theorem}\label{exactmse}
Assume $\mathrm{Rank}(\x) \geq d$, then the MSE (\ref{mse}) equals
\begin{equation}
\mathbb{E}_{\vec{\phi}}[\mathbb{E}_{\varepsilon}[\|\x \beta -\x \cls \|_2^2]]=\sigma^2 d+\sum_{i=1}^p \beta_i^2 \lambda_i \Big( 1-\frac{\lambda_i}{\eta_i} \Big).
\end{equation}
Furthermore we have 
\begin{equation}
\label{sumvar}
\sum_{i=1}^p \frac{\lambda_i}{\eta_i}=d.
\end{equation}
\end{theorem}
\begin{proof}
See Appendix.
\end{proof}
By comparing coefficients in Theorems 2 and 3, we obtain the following corollary.
\begin{corollary}
Assume $\mathrm{Rank}(\x) \geq d$, then
\begin{equation}
\forall i \in \{1,...,p\}: \,\,\, 1-\frac{\lambda_i}{\eta_i} \leq w_i
\end{equation}
\end{corollary}

As already mentioned in general we cannot give a closed-form expression for the terms $\eta_i$ in general. However, for some special cases (\ref{sumvar}) can help us to get to an exact form of the MSE of CLSE. If we assume orthonormal design ($\Sigma=C I_{p \times p}$) then we have that 
$\lambda_i/\eta_i$ is a constant for all $i$ and 
and thus, by (\ref{sumvar}), we have $\eta_i=C p /d$. This gives
\begin{equation}
\mathbb{E}_{\vec{\phi}}[\mathbb{E}_{\varepsilon}[\|\x \beta -\x \cls \|_2^2]]=\sigma^2 d+C\sum_{i=1}^p \beta_i^2  \Big( 1-\frac{d}{p}\Big),
\end{equation}
and thus we end up with a closed form MSE for this special case. 

Providing the exact mean-squared errors allows us to quantify the conservativeness of the upper bounds.
The upper bound has been shown to give a good approximation for small dimensions $d$ of the projection and for the signal contained in the larger eigenvalues. 

\section{Averaged Compressed Least Squares}
\label{sec:3}

We have  so far looked only into compressed least squares estimator with one single random projection. An issue in practice of the compressed least squares estimator is its variance due to the random projection as an additional source of randomness. This variance can be reduced by averaging multiple compressed least squares estimates coming from different random projections.
In this section we will show some properties of the averaged compressed least squares (ACLSE) estimator and discuss its advantage over the CLSE.
\begin{definition} (ACLSE) Let $\{\vec{\phi}_1,...,\vec{\phi}_K\} \, \in \mathbb{R}^{p \times d}$ be independent random projections, and let $\hat{\beta}_d^{\vec{\phi}_i} $ for all $ i \in \{1,...,K\}$ be the respective compressed least squares estimators. We define the averaged compressed least squares estimator (ACLSE) as
\begin{equation}
\hat{\beta}_d^K \; := \frac{1}{K}\sum_{i=1}^K \hat{\beta}_d^{\vec{\phi}_i}.
\end{equation}
\end{definition}
One major advantage of this estimator is that it can be calculated in parallel with the minimal number of two communications, one to send the data and one to receive the result. This means that the asymptotic computational cost of $\hat{\beta}_d^K$ is equal to the cost of $\cls$ if calculations are done on $K$ different processors.
To investigate the MSE of $\hat{\beta}_d^K$, we restrict ourselves for simplicity  to the limit case 
\begin{equation}
\hat{\beta}_d=\lim_{K \rightarrow \infty} \hat{\beta}_d^K
\end{equation}
and instead only investigate $\hat{\beta}_d$. The reasoning being that for large enough values of $K$ (say $K>100$) the behaviour of $\hat{\beta}_d$ is very similar to $\hat{\beta}_d^K$. The exact form of the MSE in terms of the $\eta_i$'s is given in ~\citet{kaban:2014}. Here we build on these results and give an explicit upper bound for the MSE.
\begin{theorem} \label{aclsemse}
Assume $\mathrm{Rank}(\x) \geq d$. Define 
\begin{equation*}
\tau=\sum_{i=1}^p \Big(\frac{\lambda_i}{\eta_i} \Big)^2 .
\end{equation*}
The MSE of $\hat{\beta}_d$ can be bounded from above by
\begin{equation*}
\mathbb{E}_{\vec{\phi}}[\mathbb{E}_{\varepsilon}[\|\x \beta-\x\hat{\beta}_d \|_2^2]]\leq\sigma^2  \tau+\sum_{i=1}^p \beta_i^2 \lambda_i w_i^2,
\end{equation*}
where the $w_i$'s are given (as in Theorem \ref{clsmse}) by
\begin{equation*}
w_i = \frac{(1+1/d)\lambda_i^2+(1+2/d)\lambda_i \operatorname{trace}(\Sigma)+\operatorname{trace}(\Sigma)^2/d}{(d+2+1/d)\lambda_i^2+2(1+1/d)\lambda_i \operatorname{trace}(\Sigma)+\operatorname{trace}(\Sigma)^2/d}.
\end{equation*}
and
\begin{equation*}
\tau \in [d^2/p,d].
\end{equation*}
\end{theorem}
\begin{proof}
See Appendix.
\end{proof}
Comparing averaging to the case where we only have one single estimator we see that there are two differences: First the variance due to the model noise $\varepsilon$ turns into $\sigma^2 \tau$ with $\tau \in [d^2/p,d]$, thus $\tau \leq d$. Secondly the shrinkage factors $w_i$ in the bias are now squared, which in total means that the MSE of $\hat{\beta}_d$ is always smaller or equal to the MSE of a single estimator $\cls$. \\
\begin{figure}[t]
\centerline{\mbox{\subfigure{\includegraphics[width=5.3cm]{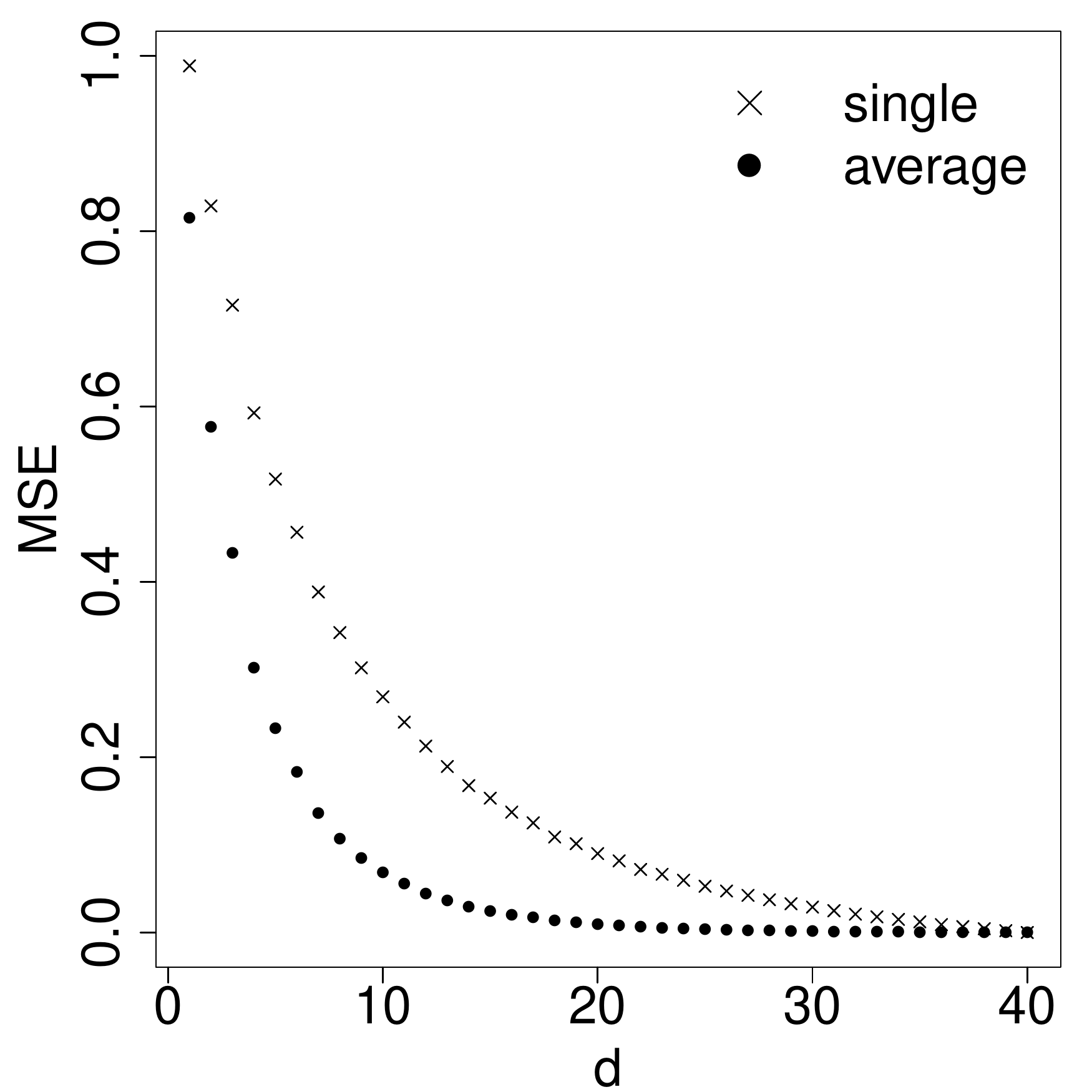}}\quad
\subfigure{\includegraphics[width=5.3cm]{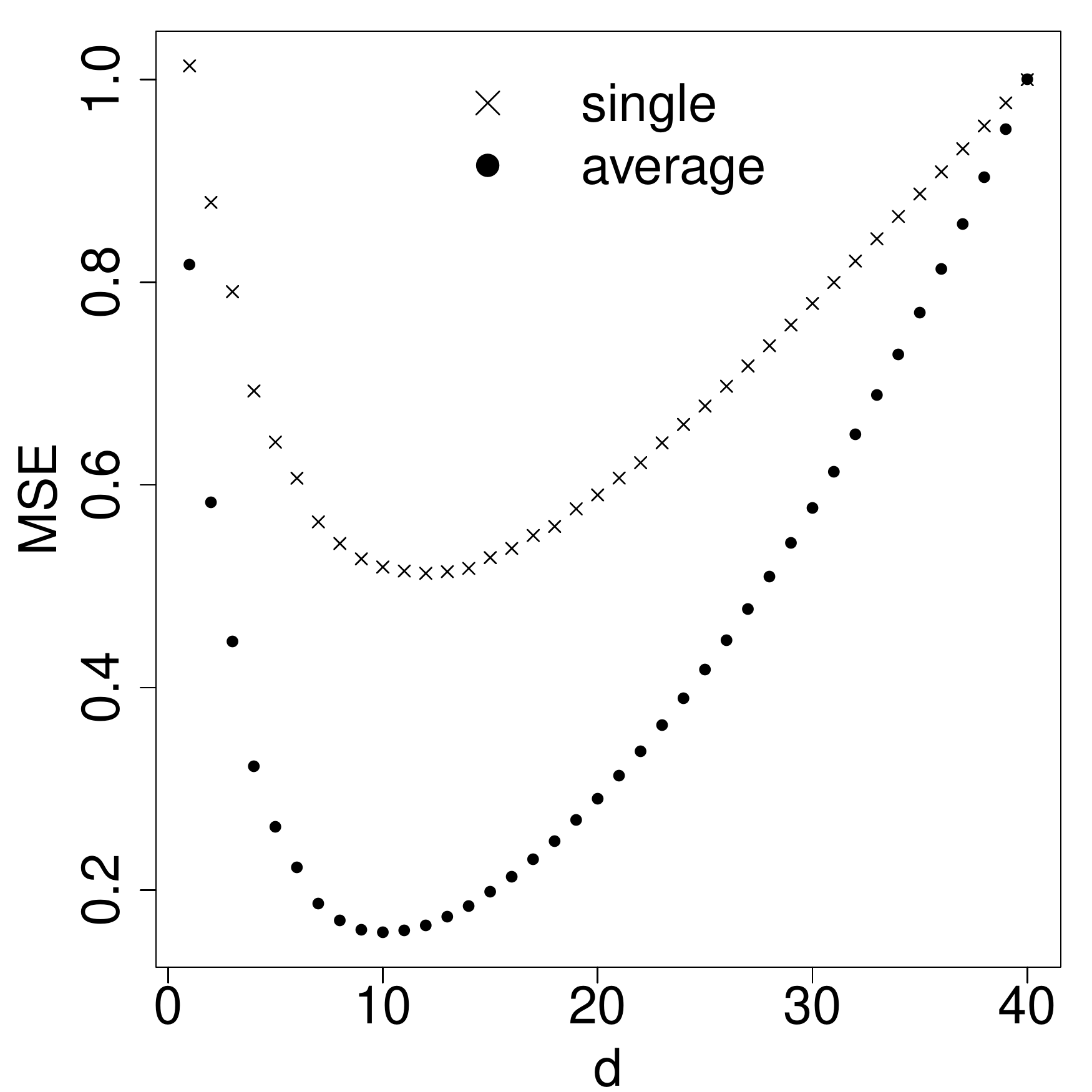}}\quad
\subfigure{\includegraphics[width=5.3cm]{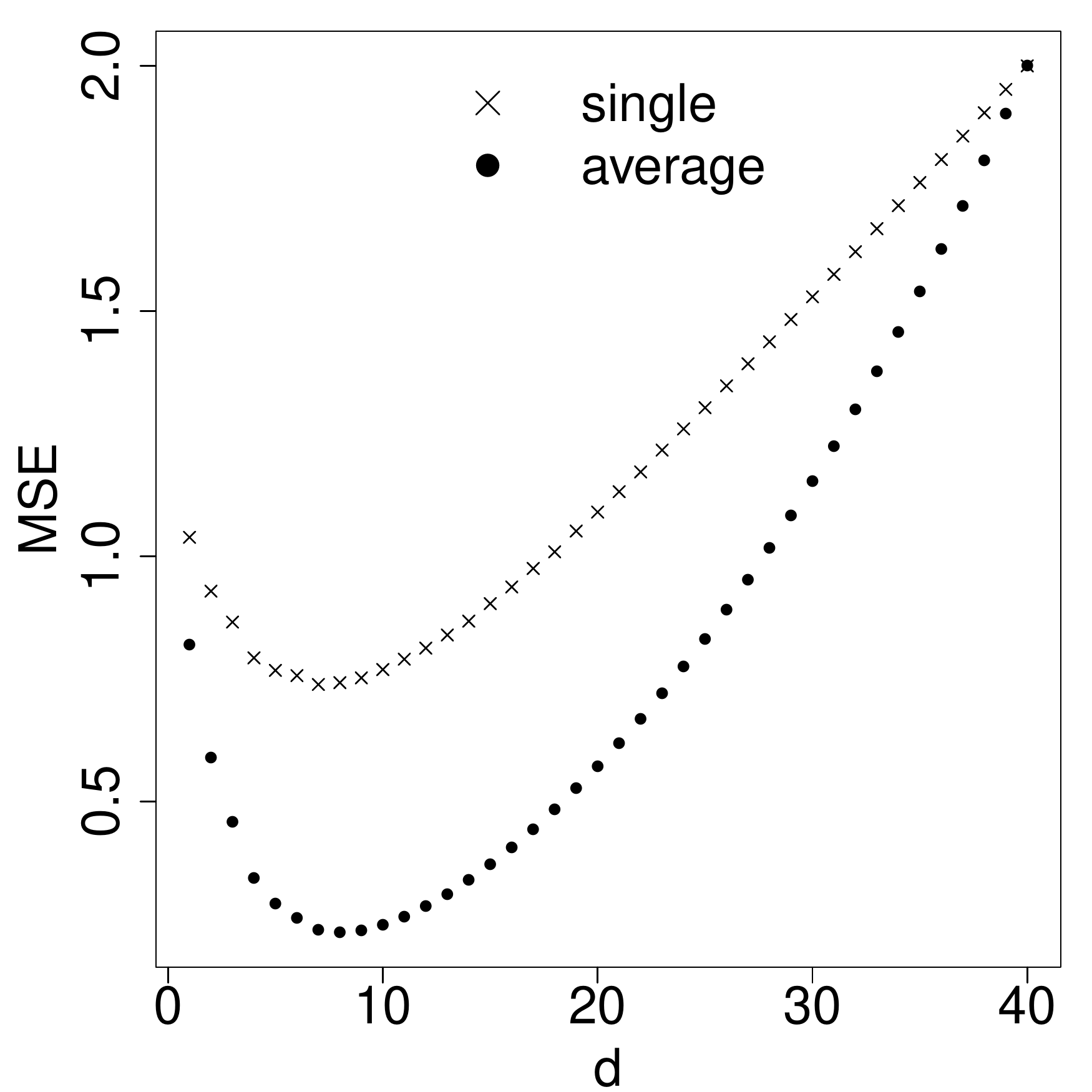}}}}
\caption{MSE of averaged compressed least squares (circle) versus  the MSE of the single estimator (cross) with covariance matrix $\Sigma_{i,i}=1/i$. On the left with $\sigma^2=0$ (only bias), in the middle $\sigma^2=1/40$ and on the left $\sigma^2=1/20$. One can clearly see the quadratic improvement in terms of MSE as predicted by Theorem~4.} \label{fig2}
\end{figure}
We investigate the behavior of $\tau$ as a function of $d$ in three different situations (Figure~\ref{fig2}). 
We first look at two extreme cases of covariance matrices for which the respective upper and lower bounds $[d^2/p,d]$ for $\tau$ are achieved. 
For the lower bound, let $\Sigma=I_{p \times p}$ be orthonormal.  Then $\lambda_i / \eta_i = c$ for all $i$, as above. From 
\begin{equation*}
\sum_{i=1}^p \lambda_i / \eta_i = d 
\end{equation*}
we get $\lambda_i/\eta_i=d/p$. This leads to
\begin{equation*}
\tau=\sum_{i=1}^p (\lambda_i / \eta_i )^2 = p \frac{d^2}{p^2} = \frac{d^2}{p},
\end{equation*}
which reproduces the lower bound.
 
\begin{figure}[t]
\centerline{\mbox{\subfigure{\includegraphics[width=5.3cm]{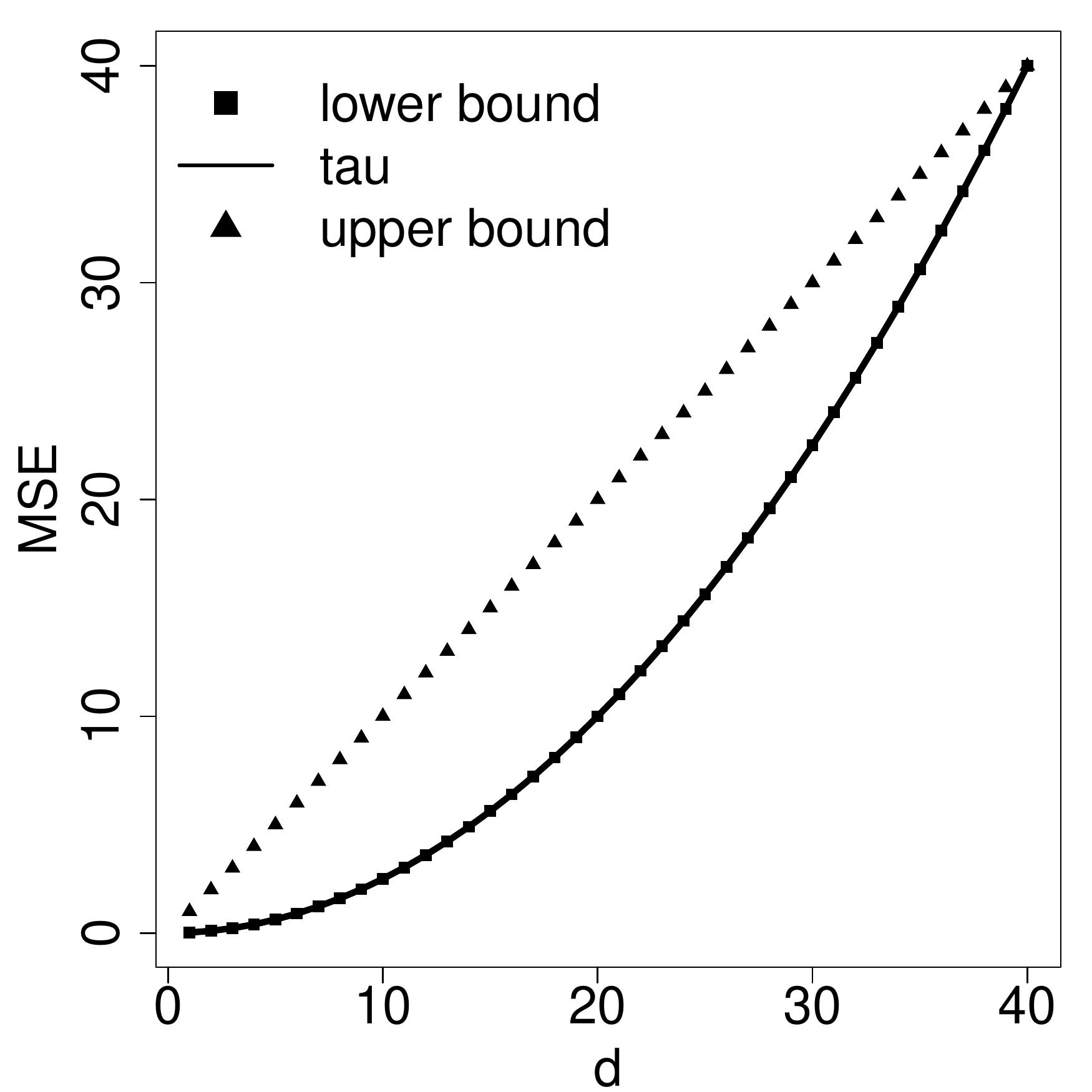}}\quad
\subfigure{\includegraphics[width=5.3cm]{varianceharm.pdf}}\quad
\subfigure{\includegraphics[width=5.3cm]{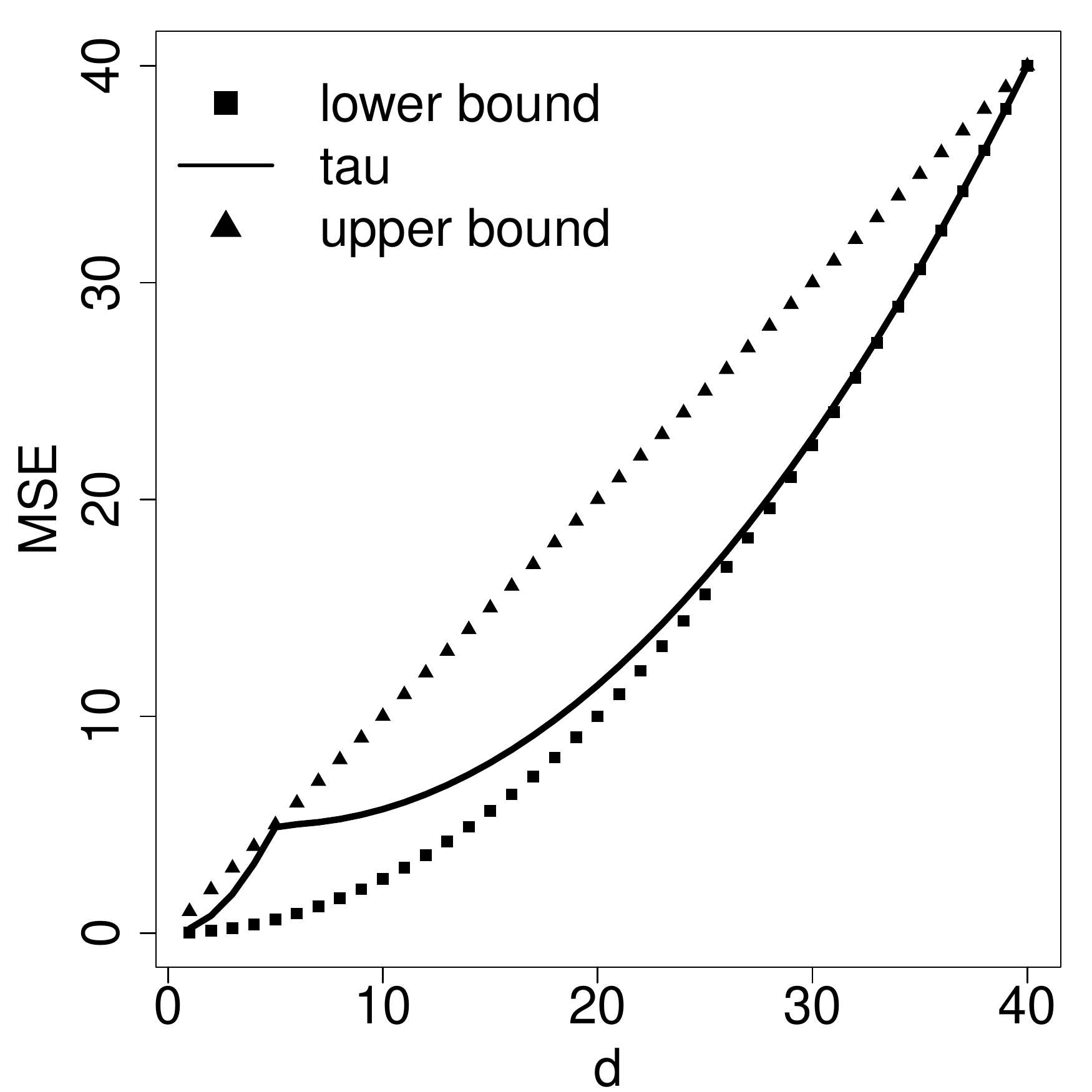}}}}
\caption{Simulations of the variance factor $\tau$ (line) as a function of $d$ for three different covariance matrices and in lower bound ($d^2/p$) and upper bound ($d$) (square, triangle). On the left ($\Sigma=I_{p \times p}$) $\tau$ as proven reaches the lower bound. In the middle ($\Sigma_{i,i}=1/i$) $\tau$ reaches almost the lower bound, indicating that in most practical examples $\tau$ will be very close to the lower bound and thus averaging improves MSE substantially compared to the single estimator. On the right the extreme case example from (\ref{specialsigma}) with $d=5$, where $\tau$ reaches the upper bound for $d=5$. } \label{fig3}
\end{figure}

We will not be able to reproduce the upper bound exactly for all $d \leq p$. But we can show that for any $d$ there exists a covariance matrix $\Sigma$, such that the upper bound is reached. The idea is to consider a covariance matrix that has equal variance in the first $d$ direction and almost zero in the remaining $p-d$. Define the diagonal covariance matrix
\begin{equation}
\label{specialsigma}
\Sigma_{i,j}=\begin{cases}
1, & \text{if } \, i=j \text{ and } i \leq d \\
\epsilon, & \text{if } \, i=j \text{ and } i > d \\
0, & \text{if } \, i \neq j \\
\end{cases}.
\end{equation} 
We show $\lim_{\epsilon \rightarrow 0}\tau=d$. For this decompose $\Phi$ into two matrices $\Phi_d \in \mathbb{R}^{d \times d}$ and $\Phi_r \in \mathbb{R}^{(p-d) \times d}$:
\begin{equation*}
\Phi= \begin{pmatrix}
  \Phi_d \\
  \Phi_r
 \end{pmatrix}.
\end{equation*}
The same way we define $\beta_d$, $\beta_r$, $\x_d$ and $\x_r$. Now we bound the approximation error of $\hat{\beta}_d^{\Phi}$ to extract information about $\lambda_i / \eta_i$. Assume a squared data matrix ($n=p$) $\x=\sqrt{\Sigma}$, then
\begin{align*}
\mathbb{E}_{\Phi}[\underset{\gamma \in \mathbb{R}^d}{\operatorname{argmin}}\| \x \beta-\x \Phi \gamma \|_2^2]& \leq \mathbb{E}_{\Phi}[\| \x\beta-\x\Phi \Phi_d^{-1} \beta_d\|_2^2] \\
&=\mathbb{E}_{\Phi}[\| \x_r \beta_r-\x_r \Phi_r \Phi_d^{-1} \beta_d  \|_2^2] \\
&=\epsilon \mathbb{E}_{\Phi}[\| \beta_r - \Phi_r \Phi_d^{-1} \beta_d\|_2^2]\\
&\leq \epsilon (2\|\beta_r \|_2^2+2\|\beta_d \|_2^2\mathbb{E}_{\Phi}[\|\Phi_r \|_2^2]\mathbb{E}_{\Phi}[\|\Phi_d^{-1} \|_2^2])\\
&\leq \epsilon C,
\end{align*}
where $C$ is independent of $\epsilon$ and bounded since the expectation of the smallest and largest singular values of a random projection is bounded. This means that the approximation error decreases to zero as we let $\epsilon \rightarrow 0$. Applying this to the closed form for the MSE of $\hat{\beta}_d^{\Phi}$ we have that 
\begin{equation*}
\sum_{i=1}^p \beta_i^2 \lambda_i \Big(1-\frac{\lambda_i}{\eta_i}\Big)\leq \sum_{i=1}^d \beta_i^2 \Big(1-\frac{\lambda_i}{\eta_i}\Big) +\epsilon \sum_{i=d+1}^p \beta_i^2 \Big(1-\frac{\lambda_i}{\eta_i}\Big)
\end{equation*}
has to go to zero as $\epsilon \rightarrow 0$, which in turn implies
\begin{equation*}
\lim_{\epsilon \rightarrow 0}\;  \sum_{i=1}^d \beta_i^2 \Big(1-\frac{\lambda_i}{\eta_i}\Big) = 0,
\end{equation*}
and thus $\lim_{\epsilon \rightarrow 0} \; \lambda_i/\eta_i =1$ for all $i \in \{1,...,d\}$. 
This finally yields a limit 
\begin{equation*}
\lim_{\epsilon \rightarrow 0}\; \sum_{i=1}^p \frac{\lambda_i^2}{\eta_i^2}=d.
\end{equation*}

This illustrates that the lower bound $d^2/p$ and upper bound $d$ for the variance factor $\tau$ can both be attained.  
Simulations suggest that $\tau$ is usually close to the lower bound, where the variance of the estimator is reduced by a factor $d/p$ compared to a single iteration of a compressed least-squares estimator, which is on top of the reduction in the bias error term. This shows, perhaps unsurprisingly, that averaging over random projection estimators improves the mean-squared error in a Rao-Blackwellization sense. We have quantified the improvement. In practice, one would have to decide whether to run multiple versions of a compressed least-squares regression in parallel or run a single random projection with a perhaps larger embedding dimension. The computational effort and statistical error tradeoffs will depend on the implementation but the bounds above will give a good basis for a decision.


\section{Discussion}
\label{sec:discuss}

We discussed some known results about the properties of compressed least-squares estimation and proposed possible tighter bounds and exact results for the mean-squared error. While the exact results do not have an explicit representation, they allow nevertheless to quantify the conservative nature of the upper bounds on the error. Moreover, the shown results allow to show a strong similarity of the error of compressed least squares, ridge  and principal component regression. We also discussed the advantages of a form of Rao-Blackwellization, where multiple compressed least-square estimators are averaged over multiple random projections. The latter averaging procedure also allows to compute the estimator trivially in a distributed way and is thus often better suited for large-scale regression analysis. The averaging methodology also motivates the use of compressed least squares in the high dimensional setting where it performs similar to ridge regression and the use of multiple random projection will reduce the variance and result in a non random estimator in the limit, which presents a computationally attractive alternative to ridge regression. 

\bibliographystyle{apalike}
\bibliography{myrefs} 

\begin{thebibliography}{}

\bibitem[Achlioptas, 2003]{achlioptas2003}
Achlioptas, D. (2003).
\newblock {Database-friendly random projections: Johnson-Lindenstrauss with
  binary coins}.
\newblock {\em Journal of Computer and System Sciences}.

\bibitem[Ailon and Chazelle, 2006]{FJLT}
Ailon, N. and Chazelle, B. (2006).
\newblock {Approximate nearest neighbors and the fast Johnson-Lindenstrauss
  transform}.
\newblock {\em Proceedings of the 38th Annual ACM Symposium on Theory of
  Computing}.

\bibitem[Blocki et~al., 2012]{Blocki:2012}
Blocki, J., Blum, A., Datta, A., and Sheffet, O. (2012).
\newblock The johnson-lindenstrauss transform itself preserves differential
  privacy.
\newblock In {\em Foundations of Computer Science (FOCS), 2012 IEEE 53rd Annual
  Symposium on}, pages 410--419. IEEE.

\bibitem[Cook, 1977]{Cook:1977}
Cook, R.~D. (1977).
\newblock Detection of influential observation in linear regression.
\newblock {\em Technometrics}, 19:15--18.

\bibitem[Dasgupta and Gupta, 2003]{jl-gauss-proof}
Dasgupta, S. and Gupta, A. (2003).
\newblock {An Elementary Proof of a Theorem of Johnson and Lindenstrauss}.
\newblock {\em Random Structures and Algorithms}.

\bibitem[Dhillon et~al., 2013a]{Dhillon:2013wz}
Dhillon, P., Lu, Y., Foster, D.~P., and Ungar, L. (2013a).
\newblock {New Subsampling Algorithms for Fast Least Squares Regression}.
\newblock In {\em Advances in Neural Information Processing Systems}.

\bibitem[Dhillon et~al., 2013b]{Dhillon:2013tw}
Dhillon, P.~S., Foster, D.~P., and Kakade, S. (2013b).
\newblock {A Risk Comparison of Ordinary Least Squares vs Ridge Regression}.
\newblock {\em The Journal of Machine Learning Research}, 14.

\bibitem[Indyk and Motwani, 1998]{jl-gauss-proof2}
Indyk, P. and Motwani, R. (1998).
\newblock {Approximate nearest neighbors: towards removing the curse of
  dimensionality}.
\newblock {\em Proceedings of the 30th Annual ACM Symposium on Theory of
  Computing}.

\bibitem[Johnson and Lindenstrauss, 1984]{jl-orig}
Johnson, W. and Lindenstrauss, J. (1984).
\newblock {Extensions of Lipschitz Mappings into a Hilbert Space}.
\newblock {\em Contemporary Mathematics: Conference on Modern Analysis and
  Probability}.

\bibitem[Kab{\'a}n, 2014]{kaban:2014}
Kab{\'a}n, A. (2014).
\newblock New bounds on compressive linear least squares regression.
\newblock In {\em Artificial Intelligence and Statistics}.

\bibitem[Lu et~al., 2013]{Lu:2013}
Lu, Y., Dhillon, P., Foster, D.~P., and Ungar, L. (2013).
\newblock Faster ridge regression via the subsampled randomized hadamard
  transform.
\newblock In {\em Advances in Neural Information Processing Systems 26}, pages
  369--377.

\bibitem[Mahoney and Drineas, 2009]{mahoney2009cur}
Mahoney, M.~W. and Drineas, P. (2009).
\newblock Cur matrix decompositions for improved data analysis.
\newblock {\em Proceedings of the National Academy of Sciences},
  106(3):697--702.

\bibitem[Maillard and Munos, 2009]{clsr}
Maillard, O.-A. and Munos, R. (2009).
\newblock {Compressed Least-Squares Regression}.
\newblock {\em NIPS}.

\bibitem[Marzetta et~al., 2011]{marzetta:2013}
Marzetta, T., Tucci, G., and Simon, S. (2011).
\newblock A random matrix-theoretic approach to handling singular covariance
  estimates.
\newblock {\em IEEE Trans. Information Theory}.

\bibitem[McWilliams et~al., 2014a]{mcwilliams2014loco}
McWilliams, B., Heinze, C., Meinshausen, N., Krummenacher, G., and
  Vanchinathan, H.~P. (2014a).
\newblock Loco: Distributing ridge regression with random projections.
\newblock {\em arXiv preprint arXiv:1406.3469}.

\bibitem[McWilliams et~al., 2014b]{McWilliams:2014}
McWilliams, B., Krummenacher, G., {Lu\v{c}i\'c}, M., and Buhmann, J.~M.
  (2014b).
\newblock Fast and robust least squares estimation in corrupted linear models.
\newblock In {\em NIPS}.

\bibitem[Tropp, 2010]{Tropp:2010uo}
Tropp, J.~A. (2010).
\newblock {Improved analysis of the subsampled randomized Hadamard transform}.
\newblock arXiv:1011.1595v4 [math.NA].

\bibitem[Zhang et~al., 2012]{zhang2012recovering}
Zhang, L., Mahdavi, M., Jin, R., Yang, T., and Zhu, S. (2012).
\newblock Recovering optimal solution by dual random projection.
\newblock {\em arXiv preprint arXiv:1211.3046}.

\bibitem[Zhou et~al., 2007]{Zhou:2007}
Zhou, S., Lafferty, J.~D., and Wasserman., L.~A. (2007).
\newblock {Compressed regression}.
\newblock {\em NIPS}.

\end{thebibliography}


\section{Appendix}

In this section we give proofs of the statements from the section theoretical results.\\
\textbf{Theorem \ref{clsmse}.} ~\citep{kaban:2014} Assume fixed design and $\mathrm{Rank}(\x) \geq d$, then the AMSE \ref{mse} can be bounded above by
\begin{equation} 
\mathbb{E}_{\vec{\phi}}[\mathbb{E}_{\varepsilon}[\|\x \beta -\x \cls \|_2^2]] \leq \sigma^2 d + \frac{\| \x \beta\|_2^2}{d}+\operatorname{trace}(\x^{\prime}\x)\frac{\| \beta \|_2^2}{d}.
\end{equation}
\begin{proof} (Sketch)
\begin{align*}
\mathbb{E}_{\vec{\phi}}[\mathbb{E}_{\varepsilon}[\|\x\beta-\x \hat{\beta}_{d}^{\vec{\phi}}\|_2^2 ]]&= \mathbb{E}_{\vec{\phi}}[\| \x \beta -\x\vec{\phi} (\vec{\phi}^{\prime} \x^{\prime} \x \vec{\phi})^{-1} \vec{\phi}^{\prime} \x^{\prime} \x \beta \|_2^2]+\sigma^2 d \\
&\leq \mathbb{E}_{\vec{\phi}}[ \| \x \beta -\x\vec{\phi} (\vec{\phi}^{\prime} \x^{\prime} \x \vec{\phi})^{-1} \vec{\phi}^{\prime} \x^{\prime} \x \vec{\phi}\vec{\phi}^{\prime} \beta \|_2^2] +\sigma^2 d \\
&=\mathbb{E}_{\vec{\phi}}[\| \x \beta -\x\vec{\phi}\vec{\phi}^{\prime} \beta \|_2^2]+\sigma^2 d.
\end{align*}
Finally a rather lengthy but straightforward calculation leads to
\begin{equation}
\mathbb{E}_{\vec{\phi}}[\| \x \beta -\x\vec{\phi}\vec{\phi}^{\prime} \beta \|_2^2]=\frac{\| \x \beta\|_2^2}{d}+\operatorname{trace}(\x^{\prime}\x)\frac{\| \beta \|_2^2}{d}
\end{equation}
and thus proving the statement above. \qed
\end{proof}
\textbf{Theorem \ref{clrmseimp}.} Assume $\mathrm{Rank}(\x) \geq d$, then the AMSE (\ref{mse}) can be bounded above by
\begin{equation} 
\mathbb{E}_{\vec{\phi}}[\mathbb{E}_{\varepsilon}[\|\x \beta -\x \cls \|_2^2]] \leq \sigma^2 d +\sum_{i=1}^p \beta_i^2 \lambda_i w_i
\end{equation}
where 
\begin{equation}
w_i = \frac{(1+1/d)\lambda_i^2+(1+2/d)\lambda_i \operatorname{trace}(\Sigma)+\operatorname{trace}(\Sigma)^2/d}{(d+2+1/d)\lambda_i^2+2(1+1/d)\lambda_i \operatorname{trace}(\Sigma)+\operatorname{trace}(\Sigma)^2/d}.\\
\end{equation}
\begin{proof}
We have for all $v \in \mathbb{R}^p$
\begin{equation*}
\mathbb{E}_{\vec{\phi}}[\underset{\hat{\gamma} \in \mathbb{R}^d}{\min}\|\x \beta -\x \vec{\phi} \hat{\gamma} \|_2^2] \leq
\mathbb{E}_{\vec{\phi}}[\|\x \beta -\x \vec{\phi}\vec{\phi}^{\prime} v\|_2^2].
\end{equation*}
Which we can minimize over the whole set $\mathbb{R}^p$:
\begin{equation*}
\mathbb{E}_{\vec{\phi}}[\underset{\hat{\gamma} \in \mathbb{R}^d}{\min}\|\x \beta -\x \vec{\phi} \hat{\gamma} \|_2^2] \leq
\underset{v \in \mathbb{R}^p}{\min}\mathbb{E}_{\vec{\phi}}[\|\x \beta -\x \vec{\phi} \vec{\phi}^{\prime} v\|_2^2].
\end{equation*}
This last expression we can calculate following the same path as in Theorem 1:
\begin{align*}
\mathbb{E}_{\vec{\phi}}[\|\x \beta -\x \vec{\phi} \vec{\phi}^{\prime}v \|_2^2]=&\beta^{\prime} \x^{\prime} \x \beta - 2 \beta^{\prime} \x^{\prime} \x \mathbb{E}_{\vec{\phi}}[\vec{\phi} \vec{\phi}^{\prime}]v\\
&+v^{\prime} \mathbb{E}_{\vec{\phi}}[\vec{\phi} \vec{\phi}^{\prime} \x^{\prime} \x \vec{\phi} \vec{\phi}^{\prime}] v\\
=&\beta^{\prime} \x^{\prime} \x \beta - 2 \beta^{\prime} \x^{\prime} \x v\\
&+(1+1/d)v^{\prime} \x^{\prime} \x  v+\frac{\operatorname{trace}(\Sigma)}{d} \|v \|_2^2,
\end{align*}
where $\Sigma=X^{\prime} X$. Next we minimize the above expression w.r.t $v$. For this we take the derivative w.r.t. $v$ and then we zero the whole expression. This yields
\begin{equation*}
2\Big(1+\frac{1}{d}\Big) \Sigma v +2\frac{\operatorname{trace}(\Sigma)}{d} I_{p \times p} v-2  \Sigma \beta=0.
\end{equation*}
Hence we have
\begin{equation*}
v=\Big(\Big(1+\frac{1}{d}\Big)\Sigma+\frac{\operatorname{trace}(\Sigma)}{d} I_{p \times p}\Big)^{-1} \Sigma \beta,
\end{equation*}
which is element wise equal to
\begin{equation*}
v_i=\frac{\beta_i \lambda_i}{(1+1/d)\lambda_i+\operatorname{trace}(\Sigma)/d}.
\end{equation*}
Define the notation $s=\operatorname{trace}(\Sigma)$. We now plug this back into the original expression and get
\begin{align*}
\underset{v \in \mathbb{R}^p}{\min} \mathbb{E}_{\vec{\phi}}[\|\x \beta -\x \vec{\phi} \vec{\phi}^{\prime}v \|_2^2]=&\beta^{\prime} \Sigma \beta - 2 \beta^{\prime} \Sigma v\\
&+(1+1/d)v^{\prime} \Sigma  v+\frac{s}{d} \|v \|_2^2 \\
=&\sum_{i=1}^p \beta_i^2 \lambda_i-2\beta_i v_i \lambda_i +(1+1/d) v_i^2 \lambda_i+s/d v_i^2 \\
=&\sum_{i=1}^p \Big(\beta_i^2 \lambda_i-2\beta_i^2 \lambda_i \frac{\lambda_i}{(1+1/d)\lambda_i+s/d}\\
&+ \beta_i^2 \lambda_i (1+1/d)\frac{\lambda_i^2}{((1+1/d)\lambda_i+s/d)^2}\\
&+\beta_i^2 \lambda_i \frac{s}{d}\frac{\lambda_i}{((1+1/d)\lambda_i+s/d)^2} \Big)\\
=&\sum_{i=1}^p \beta_i^2 \lambda_i  w_i,
\end{align*} 
by combining the summands we get for $w_i$ the expression mentioned in the theorem. \qed
\end{proof}
\textbf{Theorem \ref{exactmse}.}
Assume $\mathrm{Rank}(\x) \geq d$, then the MSE (\ref{mse}) equals
\begin{equation}
\mathbb{E}_{\vec{\phi}}[\mathbb{E}_{\varepsilon}[\|\x \beta -\x \cls \|_2^2]]=\sigma^2 d+\sum_{i=1}^p \beta_i^2 \lambda_i \Big( 1-\frac{\lambda_i}{\eta_i} \Big).
\end{equation}
Furthermore we have 
\begin{equation}
\label{sumvar}
\sum_{i=1}^p \frac{\lambda_i}{\eta_i}=d.
\end{equation}
\begin{proof}
Calculating the expectation yields
\begin{equation*}
\mathbb{E}_{\vec{\phi}}[\mathbb{E}_{\varepsilon}[\|\x \beta-\x\hat{\beta}_d \|_2^2]]=\beta^{\prime} \Sigma \beta-2\beta^{\prime} \Sigma T_d^{\phi}\Sigma \beta+\mathbb{E}_{\vec{\phi}}[\mathbb{E}_{\varepsilon}[Y^{\prime}\x \phi_d^{\x} \x^{\prime} Y]].
\end{equation*}
Going through these terms we get:
\begin{align*}
&\beta^{\prime} \Sigma \beta=\sum_{i=1}^p \beta_i^2 \lambda_i\\
&\beta^{\prime} \Sigma T_d^{\phi}\Sigma \beta=\sum_{i=1}^p \beta_i^2 \frac{\lambda_i^2}{\eta_i}\\
&\mathbb{E}_{\vec{\phi}}[\mathbb{E}_{\varepsilon}[Y^{\prime}\x \phi_d^{\x} \x^{\prime} Y]]=\beta^{\prime}\Sigma \mathbb{E}_{\vec{\phi}}[\phi_d^{\x} ]\Sigma \beta +\mathbb{E}_{\vec{\phi}}[\mathbb{E}_{\varepsilon}[\varepsilon^{\prime}\x \phi_d^{\x} \x^{\prime} \varepsilon]].
\end{align*}
The first term in the last line equals $\sum_{i=1}^p \beta_i^2 \lambda_i^2 / \eta_i$. The second can be calculated in two ways, both relying on the shuffling property of the trace operator:
\begin{align*}
\mathbb{E}_{\vec{\phi}}[\mathbb{E}_{\varepsilon}[\varepsilon^{\prime}\x \phi_d^{\x} \x^{\prime} \varepsilon]]&=\mathbb{E}_{\varepsilon}[\varepsilon^{\prime}\x T_d^{\x} \x^{\prime} \varepsilon]]=\sigma^2 \operatorname{trace}(\x T_d^{\x} \x^{\prime})=\sigma^2 \operatorname{trace}(\Sigma T_d^{\x} )=\sum_{i=1}^p \frac{\lambda_i}{\eta_i}.\\
\mathbb{E}_{\vec{\phi}}[\mathbb{E}_{\varepsilon}[\varepsilon^{\prime}\x \phi_d^{\x} \x^{\prime} \varepsilon]]&=\sigma^2 \mathbb{E}_{\vec{\phi}}[\operatorname{trace}(\x \phi_d^{\x} \x^{\prime})]=\sigma^2 \mathbb{E}_{\vec{\phi}}[\operatorname{trace}(\Sigma \phi_d^{\x})]=\sigma^2 \mathbb{E}_{\vec{\phi}}[\operatorname{trace}(I_{d \times d})]=\sigma^2 d.
\end{align*}
Adding the first version to the expectation from above we get the exact expected mean squared error. Setting both versions equal we get the equation 
\begin{equation*}
d=\sum_{i=1}^p \frac{\lambda_i}{\eta_i} \,\,\,\, \qed.
\end{equation*} 
\end{proof}
\textbf{Theorem \ref{aclsemse}.}
Assume $\mathrm{Rank}(\x) \geq d$, then there exists a real number $\tau \in [d^2/p,d]$ such that the AMSE of $\hat{\beta}_d$ can be bounded from above by
\begin{equation*}
\mathbb{E}_{\vec{\phi}}[\mathbb{E}_{\varepsilon}[\|\x \beta-\x\hat{\beta}_d \|_2^2]]\leq\sigma^2  \tau+\sum_{i=1}^p \beta_i^2 \lambda_i w_i^2,
\end{equation*}
where the $w_i$'s are given as
\begin{equation*}
w_i = \frac{(1+1/d)\lambda_i^2+(1+2/d)\lambda_i \operatorname{trace}(\Sigma)+\operatorname{trace}(\Sigma)^2/d}{(d+2+1/d)\lambda_i^2+2(1+1/d)\lambda_i \operatorname{trace}(\Sigma)+\operatorname{trace}(\Sigma)^2/d}
\end{equation*}
and
\begin{equation*}
\tau \in [d^2/p,d].
\end{equation*}
\begin{proof}
First a simple calculation ~\citep{kaban:2014} using the closed form solution gives the following equation:
\begin{equation}
\mathbb{E}_{\vec{\phi}}[\mathbb{E}_{\varepsilon}[\|\x \beta -\x \hat{\beta}_d \|_2^2]]=\sigma^2 \sum_{i=1}^p \Big(\frac{\lambda_i}{\eta_i} \Big)^2+\sum_{i=1}^p \beta_i^2 \lambda_i \Big(1-\frac{\lambda_i}{\eta_i}\Big)^2.
\end{equation}
Now using the corollary from the last section we can bound the second term the following way:
\begin{equation}
\Big(1-\frac{\lambda_i}{\eta_i}\Big)^2 \leq w_i^2.
\end{equation}
For the first term we write 
\begin{equation}
\tau =\sum_{i=1}^p \Big(\frac{\lambda_i}{\eta_i} \Big)^2.
\end{equation}
Now note that since $\lambda_i/\eta_i \leq 1$ we have
\begin{equation}
\Big(\frac{\lambda_i}{\eta_i}\Big)^2 \leq \frac{\lambda_i}{\eta_i} 
\end{equation}
and thus we get the upper bound by
\begin{equation}
\sum_{i=1}^p \Big(\frac{\lambda_i}{\eta_i} \Big)^2 \leq \sum_{i=1}^p \frac{\lambda_i}{\eta_i} =d.
\end{equation}
For the lower bound of $\tau$ we consider an optimisation problem. Denote $t_i= \frac{\lambda_i}{\eta_i}$, then we want to find $t \in \mathbb{R}^p$ such that
\begin{equation*}
\sum_{i=1}^p t_i^2 \textrm{ is minimal } 
\end{equation*} 
under the restrictions that
\begin{align*}
\sum_{i=1}^p t_i = d \textrm{ and } 0 \leq t_i \leq 1.
\end{align*}
The problem is symmetric in each coordinate and thus $t_i=c$. Plugging this into the linear sum gives $c=d/p$ and we calculate the quadratic term to give the result claimed in the theorem. \qed
\end{proof}

\end{document}